\documentclass{article}
\usepackage{amssymb}
\usepackage{amsmath}
\usepackage{verbatim}
\usepackage{graphicx}
\usepackage{pgfplots}
\usepackage{esint}
\usepackage{color}
\usepackage{amsmath}
\usepackage{array}
\usepackage{amsbsy}
\usepackage{mathabx}
\usepackage{amsthm}
\usepackage{amsfonts}
\usepackage{tikz}
\usepackage{tikz-dimline}
\usepackage{amscd}
\usepackage[all]{xy}
\usepackage{bbm}
\usepackage{mathrsfs}
\usepackage{tikz-cd}
\usepackage{mathtools}
\usepackage{harpoon}
\usepackage{enumitem}
\usepackage{makeidx}
\usepackage{graphicx}
\usepackage{subfiles}
\usetikzlibrary{matrix,calc,arrows}
\usepgfplotslibrary{fillbetween}

\usepackage[colorlinks=black, bookmarksdepth=2]{hyperref}
\usepackage{amsmath,amssymb,amsthm,amsfonts,amsbsy,latexsym,dsfont,color,graphicx,enumitem, upgreek, xcolor, xfrac}

\theoremstyle{plain}

\newtheorem*{solution*}{Solution}

\newcommand{\N}{\mathbb{N}}

\newcommand{\R}{\mathbb{R}}

\newcommand{\E}{\mathbb{E}}

\renewcommand{\P}{\mathbb{P}}

\newcommand{\inv}{^{-1}}

\newcommand{\dual}{^{\ast}}

\newcommand{\Acal}{\mathcal{A}}

\newcommand{\Hcal}{\mathcal{H}}

\newcommand{\Lcal}{\mathcal{L}}
\newcommand{\Mcal}{\mathcal{M}}

\newcommand{\Ocal}{\mathcal{O}}

\newcommand{\ba}{\begin{align*}}
\newcommand{\ea}{\end{align*}}

\theoremstyle{definition}
\newtheorem{definition}{Definition}
\numberwithin{definition}{section}
\numberwithin{equation}{section}

\usepackage[T1]{fontenc}
\usepackage[utf8]{inputenc}

\usepackage{amsmath}

    \newtheorem{remark}[definition]{Remark}
   \newtheorem{proposition}[definition]{Proposition}
   
   \newtheorem{corollary}[definition]{Corollary}
   \newtheorem{lemma}[definition]{Lemma}

   \newtheorem{theorem}[definition]{Theorem}

 \newtheorem*{example*}{Example}

\title{On Steiner Symmetrizations for First Exit Time Distributions}
\author{Tim Rolling}

\begin{document}

\maketitle

\begin{abstract}
Let $A_t$ be an $\alpha$-stable symmetric process, $0<\alpha\leq 2$, on $\R^d$ and $D\subset \R^d$ be a bounded domain. This paper presents a proof, based on the classical Brascamp-Lieb-Luttinger inequalities for multiple integrals, that the distribution of the first exit time of $A_t$ from $D$ increases under Steiner symmetrization. Further, it is shown that when a sequence of domains $\{D_m\}$ each contained in a ball $B$ and satisfying the $\varepsilon$-cone condition converges to a domain $D'$ with respect to the Hausdorff metric, the sequence of distributions of first exit times for Brownian motion from $D_m$ converges to the distribution of the first exit time of Brownian motion from  $D'$. 

These results will then be used to establish inequalities involving distributions of first exit times of $A_t$ from triangles and quadrilaterals. The primary application of these inequalities is verifying a conjecture from Ba\~nuelos in \cite{BanuelosTalk} for these planar domains. This extends a classical result of P\'olya and Szeg\"o in \cite{PolyaSzego} to the fractional Laplacian with Dirichlet boundary conditions.
\end{abstract}

\section{Introduction and Preliminaries}\label{Introduction}

Symmetrization techniques have been useful in establishing many isoperimetric inequalities. For example, let $d\ge2$ and $D\subset\R^d$ be an open connected region with $\partial D\neq\emptyset$; for our purposes here, we will refer to $D$ as a domain. If $\Lcal^d(D)<\infty$ where $\Lcal^k$ is the $k$-dimensional Lebesgue measure, $1\le k\le d$, define the symmetric decreasing rearrangement $D\dual$ as the ball centered at the origin $0\in\R^d$ with the same $d$-dimensional Lebesgue measure as $D$; then the classical isoperimetric inequality states that $D\dual$ has the minimum surface area of all domains with the same volume. We can also obtain generalized isoperimetric inequalities by considering quantities such as Dirichlet eigenvalues and integrals of transition densities. To give some examples related to what follows, let $p_D(t,x,y)$ be the transition density of the Dirichlet Laplacian $\frac{1}{2}\Delta$ in $D$. Then we have the well-known inequality below for $x\in D$, $t>0$ (see \cite{Luttinger1}, \cite{Luttinger2}, \cite{Luttinger3}):
\begin{equation*}
    \int_D p_D(t,x,y)dy\le\int_{D\dual} p_{D\dual}(t,0,y)dy,
\end{equation*}
which is equivalent to the probabilistic inequality
\begin{equation}\label{SDR-Prob}
    \P_x(\tau_D>t)\le\P_0(\tau_{D\dual}>t),
\end{equation}
where $\tau_D$ is the first exit time of Brownian motion $B_t$ from $D$ and $\P_x$ is the corresponding probability measure when $B_t$ starts at $x\in\R^d$. Also, in denoting $\lambda_D$ as the principal Dirichlet eigenvalue of $D$, using the well-known result (see \cite{BMW}):
\begin{equation}\label{eigenvalue}
\lambda_D=\lim_{t\rightarrow\infty}\frac{-1}{t}\log(\P_x(\tau_D>t))
\end{equation}
yields the classical Rayleigh-Faber-Krahn inequality
\begin{equation}\label{RFK}
\lambda_D\ge\lambda_{D\dual}.
\end{equation}

Another example of an isoperimetric inequality that will be relevant in what follows is the classical result from P\'olya and Szeg\"o (\cite{PolyaSzego}). Given $n\ge3$, this states that, among all $n$-sided polygons $P_n$ of fixed area $\beta>0$, the regular $n$-sided polygon $R_n$ of area $\beta$ minimizes the first Dirichlet eigenvalue; that is: 
\begin{equation}\label{Polya-Szego}
    \lambda_{P_n}\ge\lambda_{R_n}.
\end{equation}

While P\'olya and Szeg\"o proved this for $n=3$ and $n=4$, the problem remains open for $n\ge5$; however, recent work by Indrei (see \cite{Indrei}) has been done to prove this result for sufficiently large $n\ge5$ on local sets. That is, \cite{Indrei} constructs explicit $(2n-4)$-dimensional polygonal manifolds $\Mcal(n,\beta)$ and shows that there exists a computable $N\ge5$ such that for every $n\ge N$, the admissible $n$-gons are given via $\Mcal(n,\beta)$ and there exists an explicit set $\Acal_n(\beta)\subset\Mcal(n,\beta)$ such that $R_n$ has the smallest Dirichlet eigenvalue among all $n$-gons in $\Acal_n(\beta)$. 

In this paper, we give a different proof of \eqref{Polya-Szego} in the cases $n=3,4$ by looking at the corresponding first exit time distributions in light of (\ref{eigenvalue}); this was a conjecture given by Ba\~nuelos in \cite{BanuelosTalk}. More precisely, for $n=3,4$ and $\beta>0$, let $P_n$ and $R_n$ be as above. Then for $t>0$:
\begin{equation}\label{polygonDistribution}
    \sup_{x\in P_n}\P_x(\tau_{P_n}>t)\le\P_0(\tau_{R_n}>t).
\end{equation}
The idea behind proving (\ref{polygonDistribution}) relies on Steiner symmetrization, which is based off of the construction in \cite{Baernstein} as follows: decompose $x\in\R^d$ as $x=(y,z)$ where $y\in\R^k$, $z\in\R^{d-k}$, and $1\le k\le d-1$, then define the slice of $D$ through $z\in\R^{d-k}$ as:
$$D(z)\coloneqq\{y\in\R^k:(y,z)\in D\}.$$
By Fubini's Theorem, $D(z)$ is $\Lcal^k$-measurable for $\Lcal^{d-k}$-almost every $z\in\R^{d-k}$, and:
$$D=\bigcup_{z\in\R^{d-k}}(D(z)\times\{z\}).$$
From this, let $D\dual(z)$ be the $k$-dimensional ball centered at $0\in\R^k$ with the same $\Lcal^k$-measure as $D(z)$. As per \cite{Baernstein}, let $D\dual(z)=\emptyset$ if $D(z)$ is not $\Lcal^k$-measurable and $D\dual(z)=\{0\}$ when $\Lcal^k(D(z))=0$.

With the notation above, we can define Steiner symmetrization below.
\begin{definition}\label{SSDefn} For $D\subset\R^d$ a bounded domain, the Steiner symmetrization $D^\#$ of $D$ is:
$$D^\#\coloneqq\bigcup_{z\in\R^{d-k}}(D\dual(z)\times\{z\}).$$
\end{definition}

In general, we can take a Steiner symmetrization with respect to any $(d-k)$-dimensional hyperplane $\Lambda$. Further, given the $k$-dimensional slice of $D$ containing $x\in D$ orthogonal to $\Lambda$, let the orthogonal projection $x^\#$ be the point in the corresponding slice in $D^\#$ that is also on $\Lambda$. Hence every slice of $D$ orthogonal to $\Lambda$ has its symmetric decreasing rearrangement centered around $x^\#$.

Further examples of inequalities to which Steiner symmetrization can be applied not only include (\ref{RFK}) above with $D\dual$ replaced by $D^\#$, but also others stated in \cite{BetsakosE}. Such examples relevant for our purposes here are the following: let $x,y\in D$, $t>0$, $\Sigma$ be a one-dimensional line orthogonal to the $(d-1)$-dimensional hyperplane $\Lambda$ intersecting $D$, and $\Phi:\R\rightarrow\R$ be a nonconstant, convex, and increasing function with $\Phi(0)=0$. Then:
\begin{align}
    p_D(t,x,y)&\le p_{D^\#}(t,x^\#,y^\#)\nonumber\\
    \int_\Sigma\Phi(p_D(t,x,y))\Lcal^1(dy)&\le\int_\Sigma\Phi(p_{D^\#}(t,x^\#,y))\Lcal^1(dy),\label{SS-int-density}
\end{align}
where $x^\#$ is the orthogonal projection of $x$ onto $\Lambda$. In \eqref{SS-int-density}, the integration is performed for all $y\in\R^d$ on the line $\Sigma$. Letting $\Phi(x)=x$, we can obtain from (\ref{SS-int-density}) the following variant of (\ref{SDR-Prob}):
\begin{equation}\label{SS-Prob}
    \P_x(\tau_D>t)\le\P_{x^\#}(\tau_{D^\#}>t).
\end{equation}

As another example, in denoting the trace of $D$ as:
$$\text{Tr}(t,D)\coloneqq\int_Dp_D(t,x,x)dx$$
for $t>0$, \cite{Baernstein} showed that
\begin{equation}\label{traceIneq}
\text{Tr}(t,D)\le\text{Tr}(t,D\dual)
\end{equation}
and
\begin{equation}\label{traceIneq2}
\text{Tr}(t,D)\le\text{Tr}(t,D^\#)
\end{equation}
using a Brascamp-Lieb-Luttinger inequality for nonnegative measurable functions $f:\R^d\rightarrow\R$ (see Theorem 8.8, \cite{Baernstein}) and the following approximation of the heat kernel:
$$p_D^m(t,x,y)=\int_D\cdots\int_D\prod_{j=1}^m p(t_m,x_j,x_{j-1})dx_1\cdots dx_{m-1},$$
where $x_0=x$, $x_m=y$, $t_m=t/m$, and 
\begin{equation}\label{gaussianDefn}
    p(t,x,y)=(2\pi t)^{-d/2}e^{-|x-y|^2/2t}.
\end{equation}

These results are an extension of the same inequalities when the symmetric decreasing rearrangement is applied; hence the extension to Steiner symmetrization follows from fixing certain variables. Further, since the hyperplanes of symmetrization can vary, we can apply this technique on a sequence of hyperplanes that create a sequence of domains converging to a domain $D'$ with respect to the Hausdorff metric $d_\Hcal$ defined below for bounded domains $D_1,D_2$:
\begin{equation}\label{HausdorffDefn}
    d_\Hcal(D_1,D_2)\coloneqq\max\left\{\sup_{x\in D_1}d(x,D_2),\sup_{x\in D_2}d(D_1,x)\right\}.
\end{equation}
There are several equivalent definitions of $d_\Hcal$ (see Definition 2.3.13 of \cite{Henrot}); for what follows, we use the first of the definitions \eqref{HausdorffDefn} given there.

The goal of this paper is to apply Steiner symmetrization to probability distributions involving $\alpha$-stable symmetric processes where $0<\alpha\le2$. Recall that the stochastic process $\{A_t\}_{t\ge0}$ is an $\alpha$-stable symmetric process if it has stationary and independent increments, paths that are a.s. right continuous with left limits, and is stochastically continuous. That is, for $\eta>0$, $x\in\R^d$:
$$\lim_{t\rightarrow s}\P_x(|A_t-A_s|>\eta)=0.$$
Denoting $\E_x$ as the expectation corresponding to the probability measure $\P_x$, we have that $\{A_t\}_{t\ge0}$ also has the following characteristic function for $x\in\R^d$:
$$\E_x[e^{i\xi\cdot(A_t-x)}]=\exp(-t|\xi|^\alpha).$$

Our first result generalizes (\ref{SS-Prob}) by considering the distribution of the first exit time of the process $A_t$ from $D$ which we denote by $\tau_D^\alpha=\inf\{t>0:A_t\notin D\}$.

\begin{theorem}\label{MainResult1} Let $D\subset\R^d$ be a bounded domain, $D^\#$ be its Steiner symmetrization with respect to a $(d-k)$-dimensional hyperplane $\Lambda$, $1\le k\le d-1$, and $x_0^\#$ be the orthogonal projection of $x_0\in D$ onto $\Lambda$. Then for $t>0$:
\begin{equation*}
    \P_{x_0}(\tau_D^\alpha>t)\le\P_{x_0^\#}(\tau_{D^\#}^\alpha>t).
\end{equation*}
\end{theorem}

As stated above, Betsakos proved this in \cite{BetsakosE} in the case of Brownian motion by appealing to (\ref{SS-int-density}) when $\Phi(x)=x$. The approach here will make use of the Brascamp-Lieb-Luttinger inequality in \cite{Baernstein} along with Fubini's Theorem by treating $A_t$ as a subordination of Brownian motion $B_t$; more details on this approach are given in Section \ref{Chap2}. (With this reasoning, Theorem \ref{MainResult1} holds for any subordination of Brownian motion.) Hence, this proof will be more akin to that of (\ref{traceIneq}) and (\ref{traceIneq2}) in that we will need to deal with finite dimensional distributions as a product of transition densities over multiple integrals. In addition, given $x\in\R^d$, the precise representation of $x^\#$ will be established in the proof of Theorem \ref{MainResult1} on a case-by-case basis, so we will refrain from doing this here.


Our next result concerns extending the inequality in Theorem \ref{MainResult1} to a countable sequence of consecutive Steiner symmetrizations on $D$ by creating a new sequence of domains that converge to some domain $D'$ with respect to the Hausdorff metric. Before doing this, we must introduce the following from \cite{HenrotPierre}.

\begin{definition} Let $y,\xi\in\R^d$ with $\xi$ a unit vector and $\varepsilon>0$. Let $C(y,\xi,\varepsilon)$ be the cone of vertex $y$ (without its vertex), of direction $\xi$, and dimension $\varepsilon$, defined by:
$$C(y,\xi,\varepsilon)=\{z\in\R^d:\langle z-y,\xi\rangle\ge\cos(\varepsilon)|z-y|\text{ and }0<|z-y|<\varepsilon\},$$
where $\langle x_1,x_2\rangle$ denotes the dot product of $x_1,x_2\in\R^d$.

An open set $D$ has the $\varepsilon$-cone property if for every $x\in\partial D$, there exists a unit vector $\xi_x$ such that for every $y\in\overline{D}\cap B(x,\varepsilon)$, $C(y,\xi_x,\varepsilon)\subset D$.
\end{definition}

By Theorem 2.4.7 in \cite{HenrotPierre}, this is equivalent to a bounded domain $D$ being Lipschitz. Recall that a domain $D\subset\R^d$ is a Lipschitz domain if for some constants $L,a,r>0$, then for any $x_0\in\partial D$, there exist an orthogonal coordinate system with origin at $x_0=0$, a cylinder $K=K'\times(-a,a)$ centered at the origin, with $K'$ an open ball in $\R^{d-1}$ of radius $r$, and a function $\varphi:K'\rightarrow(-a,a)$, $L$-Lipschitz continuous with $\varphi(0)=0$, and:
\begin{align*}
    \partial D\cap K&=\{(x',\varphi(x')):x'\in K'\}\\
    D\cap K&=\{(x',x_d):x'\in K',x_d>\varphi(x')\}.
\end{align*}
If $D$ satisfies the $\varepsilon$-cone property, then $L=\cot\varepsilon$, $a=2\varepsilon$, and $r=\varepsilon$ as per \cite{HenrotPierre}.


Here we prove a more general result to which we will apply to specific domains in $\R^2$. In what follows, we will consider the class of domains $\Ocal_{\varepsilon,B}$ as given in \cite{HenrotPierre} where, given a ball $B\subset\R^d$:
$$\Ocal_{\varepsilon,B}\coloneqq\{D\subset \R^d: D\subset B\text{ and }D\text{ satisfies the }\varepsilon\text{-cone property}\}.$$

\begin{theorem}\label{MainResult2}
    Given $\varepsilon>0$ and a ball $B\subset\R^d$, let $\{D_m\}_{m=0}^\infty$ be a sequence of domains in $\Ocal_{\varepsilon,B}$ that converges to a domain $D'$ with respect to the Hausdorff metric. Further, suppose there exists a sequence $\{x_m\}_{m=0}^\infty$ such that $x_m\in D_m$ for every $m\ge0$ and $x_m\rightarrow x'$ for some $x'\in D'$. If $\tau_{D_m}$ and $\tau_{D'}$ denote the first exit times of Brownian motion from $D_m$ and $D'$, respectively, then for any $t>0$:
    \begin{equation*}
        \lim_{m\rightarrow\infty}\P_{x_m}(\tau_{D_m}>t)=\P_{x'}(\tau_{D'}>t).
    \end{equation*}
\end{theorem}
By Theorem 2.4.10 in \cite{HenrotPierre}, $D'$ as in the above result is also in the class $\Ocal_{\varepsilon,B}$ of domains.

The assumption that each of $D_m$ and $D'$ satisfy the $\varepsilon$-cone property is due to the fact that Theorem \ref{MainResult2} makes use of Lemma \ref{BMWGeneral}, which requires that the domains be Lipschitz. In addition, the first exit times above can be replaced by the respective first exit times of an $\alpha$-stable symmetric process provided each $D_m$ is a convex domain (note from this that $D'$ is also convex; see Section 2.2.3 of \cite{HenrotPierre}). The reasoning for this will be made clear in Section \ref{Chap3}.

In the context of the work presented here, Theorem \ref{MainResult2} will be applied to a sequence of domains $\{D_m\}_{m=0}^\infty\subset\Ocal_{\varepsilon,B}$ formed from applying $m$ Steiner symmetrizations with respect to the hyperplanes $\{\Lambda_m\}_{m=1}^\infty$ with $D_0=D$ and $D_{m}=(D_{m-1})^\#$ where  symmetrization is with respect to $\Lambda_m$. We also assume that this sequence converges to $D'$ with respect to the Hausdorff metric. Denoting the corresponding sequence of orthogonal projections as $x_m\in D_m$ with $x_0=x\in D$, we have the following corollary.

\begin{corollary}\label{MainResult2-SS} Given $\varepsilon>0$ and $B\subset\R^d$ a ball, let the sequences $\{x_m\}_{m=0}^\infty$, $\{D_m\}_{m=0}^\infty$, and $\{\Lambda_m\}_{m=1}^\infty$ be as above with $\{D_m\}_{m=0}^\infty\subset\Ocal_{\varepsilon,B}$ converging to $D'$ in the Hausdorff metric and $x_m$ converging to some $x'$. Then $x'\in D'$ and for any $t>0$:
\begin{equation*}
\lim_{m\rightarrow\infty}\P_{x_m}(\tau_{D_m}>t)=\P_{x'}(\tau_{D'}>t).
\end{equation*}
\end{corollary}

The fact that $x_m\in D_m$ converges to some $x'\in D'$ is established at the end of Section \ref{Chap3}.

The application of these results is the following from which one can use along with (\ref{eigenvalue}) to prove the P\'olya-Szeg\"o result \eqref{Polya-Szego} in the cases $n=3,4$. 

\begin{corollary}\label{applications} For $n=3,4$, let $P_n$ be an $n$-sided polygon in $\R^2$ of fixed area and $R_n$ be a regular $n$-sided polygon centered at the origin $0\in\R^2$ with the same area as $P_n$. If we have that $\tau_{P_n}^\alpha$ and $\tau_{R_n}^\alpha$ are first exit times of an $\alpha$-stable symmetric process from their respective domains, then for $t>0$:
\begin{equation}\label{distributionIneq2}
    \sup_{x\in P_n}\P_x(\tau_{P_n}^\alpha>t)\le\P_0(\tau_{R_n}^\alpha>t).
\end{equation}
\end{corollary}

While the inequality (\ref{distributionIneq2}) is identical in structure to (\ref{polygonDistribution}), the reasoning after Theorem \ref{MainResult2} above allows us to consider exit times of $\alpha$-stable symmetric processes in this case, not just Brownian motion; further details are provided in Section \ref{Chap3}.

The rest of this paper will be organized as follows. In Sections 2 and 3 we prove Theorems \ref{MainResult1} and \ref{MainResult2}, respectively. In Section 4 we prove Corollary \ref{applications} as an application of Theorems \ref{MainResult1} and \ref{MainResult2}.

\section{Proof of Theorem \ref{MainResult1}}\label{Chap2}

First, as mentioned in the previous section, for $0<\alpha\le2$, the $\alpha$-stable symmetric symmetric process $A_t$ in $\R^d$ has the representation $A_t=B_{2\sigma_t}$ where $\sigma_t$ is a stable subordinator of index $\alpha/2$ independent of the Brownian process (see \cite{SSPs}). Thus, if we denote $p^\alpha(t,x,y)=p^\alpha(t,x-y)$ and $g_\alpha(t,s)$ as the transition densities for $A_t$ and $\sigma_t$, respectively, then:
\begin{equation}\label{SSPasBM}
    p^\alpha(t,x,y)=\int_0^\infty p(s,x,y)g_{\alpha/2}(t,s)ds,
\end{equation}
where $p(s,x,y)$ is as in \eqref{gaussianDefn}.

In the calculations that follow, we will require an extra approximation by an increasing sequence of smooth domains $\{D_i\}_{i=1}^\infty$ with $\overline{D_i}\subset D_{i+1}$ and $D_i\nearrow D$ as per \cite{AizenmanSimon}. In the case of Brownian motion $B_t$, this is always required since $B_t$ is a.s. continuous and hence $\P_x(B_{\tau_D}\in\partial D)=1$ for any $x\in D$; however, for $\alpha$-stable symmetric processes $A_t$, $\alpha\in(0,2)$, this is not always required. For example, Bogdan in Lemma 6 of \cite{Bogdan} showed that for any Lipschitz domain $D$:
\begin{equation}\label{LipschitzSSP}
    \P_x(A_{\tau_D^\alpha}\in\partial D)=0,
\end{equation}
where $\tau_D^\alpha$ is the first exit time of $A_t$ from $D$ as in Section \ref{Introduction}. Wu in \cite{Wu} imposed more general conditions on $D$ for which (\ref{LipschitzSSP}) holds, yet even this does not exhaust the list of all possible domains; in fact, Wu shows there are still open sets for which the probability in (\ref{LipschitzSSP}) is positive. Because of this, we will impose this extra approximation of $D$ even on $\alpha$-stable symmetric processes.

Thus, using the right continuity of the sample paths along with the Markov Property, we obtain:
\begin{align*}
    \P_{x_0}(\tau_D^\alpha>t)&=\P_{x_0}(A_s\in D,0\le s\le t)\\
    &=\lim_{i\rightarrow\infty}\lim_{m\rightarrow\infty}\P_{x_0}(A_{jt/m}\in D_i, j=1,\dots,m)\\
    &=\lim_{i\rightarrow\infty}\lim_{m\rightarrow\infty}\int_{D_i}\cdots\int_{D_i}\prod_{j=1}^mp^\alpha(t_m,x_j,x_{j-1})dx,
\end{align*}
where $t_m=t/m$ and $dx=dx_m\cdots dx_1$; similar notation will be used for $dy$, $dz$, etc. One can use this along with \eqref{SSPasBM} and Fubini's Theorem to get that it suffices to prove Theorem \ref{MainResult1} for the case when $A_t$ is Brownian motion.

Proceeding from this, let $\tau_D$ be the first time Brownian motion exits $D$; then similar reasoning as above yields:
\begin{align}\label{BMapproximation}
    \P_{x_0}(\tau_D>t)=\lim_{i\rightarrow\infty}\lim_{m\rightarrow\infty}\int_{D_i}\cdots\int_{D_i}\prod_{j=1}^mp(t_m,x_j,x_{j-1})dx.
\end{align}

From here, we will use the Brascamp-Lieb-Luttinger inequality in Theorem 8.8 in \cite{Baernstein} and Fubini's Theorem to establish Theorem \ref{MainResult1}. Since \cite{BLL} involves the notion of the symmetric decreasing rearrangement of a function, we first recall its definition.

\begin{definition} Let $f:\R^d\rightarrow\R$ be a measurable and nonnegative function; then the symmetric decreasing rearrangement $f\dual:\R^d\rightarrow\R$ is the unique function that satisfies:
\begin{align}
    f\dual(x)&=f\dual(y)\text{ if }|x|=|y|,\label{SDRProp1}\\
    f\dual(x)&\le f\dual(y)\text{ if }|x|\ge|y|,\label{SDRProp2}\\
    \lim_{|x|\rightarrow|y|^+}&f\dual(x)=f\dual(y),\label{SDRProp3}\\
    \Lcal^d\{x:f(x)>t\}&=\Lcal^d\{x:f\dual(x)>t\},\text{ }t>0.\label{SDRProp4}
\end{align}
\end{definition}

This definition will be applied to the function $p(t,x,y)$ above; however, since $p$ is already a nonincreasing radially symmetric function about its maximum, we get that $p(t,x,x_0)$ has the symmetric decreasing rearrangement $p(t,x,0)$ for $t>0$ and $x,x_0\in D$.

Let us now return to proving Theorem \ref{MainResult1}. The idea here is to take the iterated integral expression in (\ref{BMapproximation}) and establish the following:
\begin{align}
    &\int_{D_i}\cdots\int_{D_i}\prod_{j=1}^m p(t_m,x_j,x_{j-1})dx\nonumber\\
    &\le\int_{D_i^\#}\cdots\int_{D_i^\#}p(t_m,x_1,x_{0}^\#)\prod_{j=2}^m p(t_m,x_j,x_{j-1})dx,\label{FDDIneq}
\end{align}
where the Steiner symmetrization is performed with respect to a $(d-k)$-dimensional hyperplane $\Lambda$, $1\le k\le d-1$, $x_0^\#$ is the orthogonal projection of $x_0$ defined in Section 1. From this, one can obtain:
\begin{align*}\label{BMapproximation}
    \P_{x_0}(\tau_D>t)&=\lim_{i\rightarrow\infty}\lim_{m\rightarrow\infty}\int_{D_i}\cdots\int_{D_i}\prod_{j=1}^mp(t_m,x_j,x_{j-1})dx\\
    &\le\lim_{i\rightarrow\infty}\lim_{m\rightarrow\infty}\int_{D_i^\#}\cdots\int_{D_i^\#}p(t_m,x_0^\#,x_{1})\prod_{j=2}^mp(t_m,x_j,x_{j-1})dx\\
    &=\P_{x_0^\#}(\tau_{D^\#}>t)
\end{align*}

Note that the last inequality follows since $D_i\nearrow D$ and $\overline{D_i}\subset D_{i+1}$ implies $D_i^\#\nearrow D^\#$ and $\overline{D_i^\#}\subset D_{i+1}^\#$. To see why, note that since $D_i\nearrow D$, each $k$-dimensional slice of $D$ perpendicular to the hyperplane $\Lambda$ (denote by $D_k$) satisfies $(D_k)_i\nearrow D_k$ so that their symmetric decreasing rearrangements also satisfy this property; that is, $(D_k)_i\dual\nearrow D_k\dual$. Hence by Definition \ref{SSDefn}, $D_i^\#\nearrow D^\#$. Similar reasoning also gives that $\overline{D_i}\subset D_{i+1}$ implies $\overline{D_i^\#}\subset D_{i+1}^\#$.

From here, the proof will be broken up into four cases based on the choice of the $(d-k)$-dimensional hyperplane up to rotations and translations of
\begin{equation}\label{hyperplane}
    \Lambda=\{(x^1,\dots,x^k,x^{k+1},\dots,x^d)\in\R^d:x^1=\cdots=x^k=0\}.
\end{equation}
Without loss of generality, assume the symmetrization is done with respect to the first $k$ coordinates in what follows; otherwise, permute the coordinates so that the symmetrization is done with respect to $x^1,\dots,x^k$.

In the first case, let $\Lambda$ be as in \eqref{hyperplane} so that the origin $0\in\Lambda$ and the $x^1$-$,\cdots,x^k$-axes are orthogonal to $\Lambda$. Therefore if $x_0=(y_0,z_0)$ with $y_0\in\R^k$, $z_0\in\R^{d-k}$, then $x_0^\#=(0,z_0)$ since the last $d-k$ coordinates are unaffected in the symmetric decreasing rearrangement of each $k$-dimensional slice of $D$ orthogonal to $\Lambda$ as constructed above. In what follows, we will use the shorthand notation $D^k=D(z_k)$ and $(D^k)\dual=D\dual(z_k)$ for $1\le k\le d-1$. We will also use the notation $$p^{(z_1,z_2)}(t,y_1,y_2)\coloneqq p(t,(y_1,z_1),(y_2,z_2))$$ to denote the slice of $p$ in fixing $z_1,z_2\in\R^{d-k}$ and letting $y_1,y_2\in\R^k$ vary.

With the established setup, we obtain the following using Theorem 8.8 in \cite{Baernstein} and Fubini's Theorem:
\begin{align*}
    &\int_{D_i}\cdots\int_{D_i}p\left(t_m,x_1,x_0\right)\prod_{j=2}^m p\left(t_m,x_j,x_{j-1}\right)dx\\
    &=\int_{\R^{m(d-k)}}\left(\int_{D_i^1}\cdots\int_{D_i^m} p^{(z_1,z_0)}\left(t_m,y_1,y_0\right)\prod_{j=2}^mp^{(z_j,z_{j-1})}\left(t_m,y_j,y_{j-1}\right)dy\right)dz\\
    &\le\int_{\R^{m(d-k)}}\left(\int_{(D_i^1)\dual}\cdots\int_{(D_i^m)\dual} p^{(z_1,z_0)}\left(t_m,y_1,0\right)\right.\\&\hspace{6cm}\left.\times\prod_{j=2}^mp^{(z_j,z_{j-1})}\left(t_m,y_j,y_{j-1}\right)dy\right)dz\\
    &=\int_{D_i^\#}\cdots\int_{D_i^\#}p(t_m,x_1,x_0^\#)\prod_{j=2}^m p(t_m,x_j,x_{j-1})dx,
\end{align*}
establishing (\ref{FDDIneq}) in this case.

In the next case, consider a hyperplane of the form $$\Lambda=\{(x^1,\dots,x^d)\in\R^d:x^j=\omega^j, j=1,\dots,k\}$$ with at least one $\omega^i\neq0$ so that $\Lambda$ is orthogonal to the $x^1$-$,\dots,x^k$-axes, but the origin $0\notin \Lambda$. From this, denote \begin{align*}
    \omega&\coloneqq(\omega^1,\dots,\omega^k),\\
    w=y-\omega&=(y^1-\omega^1,\dots,y^k-\omega^k),
\end{align*}
with $w^j=y^j-\omega^j$ for $j\le k$ and, for any measurable set $E\subset\R^k$:
$$E-\omega=\{y-\omega:y\in E\}.$$
Using this definition, we obtain:
$$\Lambda-\omega=\{(x^1,\dots,x^d)\in\R^d:x^1=\cdots=x^k=0\}.$$
In addition, if we decompose $D$ into slices, then we can define the following domains in $\R^d$:
$$D-\omega\coloneqq D-(\omega,0)=\bigcup_{z\in\R^{d-k}}((D(z)-\omega)\times\{z\}),$$
and:
$$(D-\omega)^\#\coloneqq (D-(\omega,0))^\#=\bigcup_{z\in\R^{d-k}}((D(z)-\omega)\dual\times\{z\}).$$
The above notation may now be applied to the iterated integral in (\ref{FDDIneq}) to get:
\begin{align*}
    &\int_{D_i}\cdots\int_{D_i}p\left(t_m,x_1,x_0\right)\prod_{j=2}^mp\left(t_m,x_j,x_{j-1}\right)dx\\
    &=\int_{\R^{m(d-k)}}\left(\int_{D_i^1}\cdots\int_{D_i^m}p^{(z_1,z_0)}\left(t_m,y_1,y_0\right)\prod_{j=2}^mp^{(z_j,z_{j-1})}\left(t_m,y_j,y_{j-1}\right)dy\right)dz\\
    &=\int_{\R^{m(d-k)}}\left(\int_{D_i^1-\omega}\cdots\int_{D_i^m-\omega}p^{(z_1,z_0)}\left(t_m,w_1,y_0-\omega\right)\right.\\&\hspace{6cm}\left.\times\prod_{j=2}^mp^{(z_j,z_{j-1})}\left(t_m,w_j,w_{j-1}\right)dw\right)dz\\
    &\le\int_{\R^{m(d-k)}}\left(\int_{(D_i^1-\omega)\dual}\cdots\int_{(D_i^m-\omega)\dual}p^{(z_1,z_0)}\left(t_m,w_1,0\right)\right.\\&\hspace{6cm}\left.\times\prod_{j=2}^mp^{(z_j,z_{j-1})}\left(t_m,w_j,w_{j-1}\right)dw\right)dz\\
    &=\int_{\R^{m(d-k)}}\left(\int_{(D_i^1-\omega)\dual+\omega}\cdots\int_{(D_i^m-\omega)\dual+\omega}p^{(z_1,z_0)}\left(t_m,y_1,\omega\right)\right.\\&\hspace{6cm}\left.\times\prod_{j=2}^mp^{(z_j,z_{j-1})}\left(t_m,y_j,y_{j-1}\right)dy\right)dz\\
    &=\int_{(D_i-\omega)^\#+\omega}\cdots\int_{(D_i-\omega)^\#+\omega}p\left(t_m,x_1,(\omega,z_0)\right)\prod_{j=2}^mp\left(t_m,x_j,x_{j-1}\right)dx
\end{align*}
so that $x_0^\#=(\omega,z_0)$ in this case.

For the third case, consider the rotation operator $\rho\neq I$ (where $I$ is the identity operator) such that $$\rho \Lambda=\{\rho x:x\in \Lambda\}=\{(x^1,\dots,x^d)\in\R^d:x^1=\cdots=x^k=0\}$$ and $\rho0=0$ so that $0\in \Lambda$, but $\Lambda$ is not orthogonal to at least one of the $x^1,\dots,x^k$-axes (such an operator exists by the invertibility of $\rho$). If we denote $\xi_m=\rho x_m$, then the fact that $\rho$ is a linear transformation gives that each $\rho(x_j-x_{j-1})=\xi_j-\xi_{j-1}$ for each $j$. Therefore, if we let $\rho D=\{\rho x:x\in D\}$ and $\rho x_0=(y_0',z_0')$ for $y_0'\in\R^k$, $z_0'\in\R^{d-k}$, we may apply the first case with respect to $\rho \Lambda$ to get:
\begin{align*}
    &\int_{D_i}\cdots\int_{D_i}p\left(t_m,x_1,x_0\right)\prod_{j=2}^mp\left(t_m,x_j,x_{j-1}\right)dx\\
    &=\int_{\rho D_i}\cdots\int_{\rho D_i}p\left(t_m,\xi_1,(y_0',z_0')\right)\prod_{j=2}^mp\left(t_m,\xi_j,\xi_{j-1}\right)d\xi\\
    &\le\int_{(\rho D_i)^\#}\cdots\int_{(\rho D_i)^\#}p\left(t_m,\xi_1,(0,z_0')\right)\prod_{j=2}^mp\left(t_m,\xi_j,\xi_{j-1}\right)d\xi\\
    &=\int_{\rho\inv((\rho D_i)^\#)}\cdots\int_{\rho\inv((\rho D_i)^\#)}p(t_m,x_1,\rho\inv(0,z_0'))\prod_{j=2}^mp(t_m,x_j,x_{j-1})dx,
\end{align*}
so that $x_0^\#=\rho\inv(0,z_0')$.

In this last case, let the hyperplane $\Lambda$ be neither orthogonal to at least one of the $x^1$-$\dots,x^k$-axes nor have the origin. Since there exists $(\omega^1,\dots,\omega^d)\in \Lambda$ such that $\omega^i=0$ for $i>k$, if we translate $\Lambda$ by the $k$-dimensional point $(\omega^1,\dots,\omega^k)$, then:
$$0\in \Lambda'\coloneqq \Lambda-(\omega^1,\dots,\omega^k,0,\dots,0)$$ 
so that we may apply the appropriate rotation operator $\rho$ to get that
$$\rho \Lambda'=\{(\gamma^1,\dots,\gamma^d)\in\R^d:\gamma^1=\cdots=\gamma^k=0\}.$$ 
Thus, if we let:
\begin{itemize}
    \item $\omega=(\omega^1,\dots,\omega^k)\in\R^k$,
    \item $x_0=(y_0,z_0)$, $\rho x_0=(y_0',z_0')$ for $y_0,y_0'\in\R^k$, $z_0,z_0'\in\R^{d-k}$ as in the third case and $\rho'\coloneqq\rho\inv(0,z_0')$,
    \item $D-\omega$, $D+\omega$ be as in the second case for $D\subset\R^d$,
    \item $x_j=(y_j,z_j)$, $x_j'=(w_j,z_j)$, where $w_j=y_j-\omega$, $y_j\in\R^k$,
\end{itemize}
we get by the third case the following:
\begin{align*}
    &\int_{D_i}\cdots\int_{D_i}p(t_m,x_1,x_0)\prod_{j=2}^mp(t_m,x_j,x_{j-1})dx\\
    &=\int_{D_i-\omega}\cdots\int_{D_i-\omega}p(t_m,x_1',(y_0-\omega,z_0))\prod_{j=2}^mp(t_m,x_j',x_{j-1}')dx'\\
    &\le{\small\int_{\rho\inv((\rho(D_i-\omega))^\#)}\cdots\int_{\rho\inv((\rho(D_i-\omega))^\#)}p(t_m,x_1',\rho')\prod_{j=2}^mp(t_m,x_j',x_{j-1}')dx'}\\
    &=\int_{\rho\inv((\rho(D_i-\omega))^\#)+\omega}\cdots\int_{\rho\inv((\rho(D_i-\omega))^\#)+\omega}p(t_m,x_1,x_0^\#)\prod_{j=2}^mp(t_m,x_j,x_{j-1})dx
\end{align*}
where $x_0^\#=\rho\inv(0,z_0')+(\omega,0)$ in this most general case. This concludes the proof of Theorem \ref{MainResult1}.

\section{Proof of Theorem \ref{MainResult2}}\label{Chap3}
To establish Theorem \ref{MainResult2}, we will first prove the following lemma which is a generalization of Lemma 5.2 from \cite{BMW}. For what follows, $\beta_D,\beta_G>0$ depend on the Lipschitz character of $D$ and $G$, respectively; more details will be given on these constants in Corollary \ref{LpConvergence}.

\begin{lemma}\label{BMWGeneral} Let $D,G\subset\R^d$ be Lipschitz domains and let $C_{p,D}$ be a constant dependent on $p>0$ and $D$, and define $C_{p,G}$ similarly.
\begin{enumerate}
    \item If $p\ge1$, then:
    \begin{align*}
        &\sup_{x\in G\cup D}\E_x[|\tau_D-\tau_G|^p]\\
        &\le\max\left\{C_{p,D}\sup_{x\in D\setminus{G}}(d(x,\partial D))^{\beta_D},C_{p,G}\sup_{x\in G\setminus{D}}(d(x,\partial G))^{\beta_G}\right\}.
    \end{align*}
    \item If $p\in(0,1)$, then:
    \begin{align*}
        &\sup_{x\in G\cup D}\E_x[|\tau_D-\tau_G|^p]\\
        &\le\max\left\{C_{1,D}^p\sup_{x\in D\setminus{G}}(d(x,\partial D))^{p\beta_D},C_{1,G}^p\sup_{x\in G\setminus{D}}(d(x,\partial G))^{p\beta_G}\right\}.
    \end{align*}
\end{enumerate}
\end{lemma}
\begin{proof} We will break this up into cases:\\
First, let $x\notin D\cup G$; then $\tau_D=\tau_G=0$ a.s., and so the inequality is trivial. From here, assume that $\tau_D\neq\tau_G$ a.s.\\
Next, let $x\in D\setminus{G}$; then $\tau_G=0$ a.s., so by the proof of Lemma 6.4 in \cite{BMW} and the fact that $D\setminus{G}=D\setminus{(D\cap G)}$, we obtain the following for $p\ge1$:
$$\sup_{x\in D\setminus{G}}\E_x[\tau_D^p]\le C_{p,D}\sup_{x\in D\setminus{G}}\E_x[\tau_D]\le C_{p,D}\sup_{x\in D\setminus{G}}(d(x,\partial D))^{\beta_D}.$$
If $p\in(0,1)$, then an application of the proof of Lemma 6.4 from \cite{BMW} along with Jensen's inequality yields:
$$\sup_{x\in D\setminus{G}}\E_x[\tau_D^p]\le\sup_{x\in D\setminus{G}}(\E_x[\tau_D])^p\le C_{1,D}^p\sup_{x\in D\setminus{G}}(d(x,\partial D))^{p\beta_D},$$
The case when $x\in G\setminus{D}$ uses the same argument as the second case above with $D,G$ interchanged.\\
Lastly, let $x\in{D\cap G}$. First let $\tau_D>\tau_G$ a.s.; that is, the Brownian process starting at $x$ first exits $G$. Then $\tau_G=\tau_{D\cap G}$ a.s., and so we may apply the result of Lemma 6.4 in \cite{BMW} to get the following when $p\ge1$:
\begin{align*}
    \sup_{x\in {D\cap G}}\E_x[|\tau_D-\tau_G|^p]&=\sup_{x\in {D\cap G}}\E_x[|\tau_D-\tau_{D\cap G}|^p]\\&\le C_{p,D}\sup_{x\in {D\setminus (D\cap G)}}(d(x,\partial D))^{\beta_D}\\
    &=C_{p,D}\sup_{x\in D\setminus {G}}(d(x,\partial D))^{\beta_D}.
\end{align*}
For $p\in(0,1)$, another application of Jensen's inequality along with the above work yields:
$$\sup_{x\in{D\cap G}}\E_x[|\tau_D-\tau_G|^p]\le\sup_{x\in {D\cap G}}(\E_x[|\tau_D-\tau_{G}|])^p\le C_{1,D}^p\sup_{x\in D\setminus {G}}(d(x,\partial D))^{p\beta_D}.$$
The case when $\tau_D<\tau_G$ a.s. is identical to the one above in interchanging $D,G$.
\end{proof}

\begin{remark}\label{generalizationRemark} It should be noted that \cite{Kulczycki} established intrinsic ultracontractivity for the semigroup of $A_t$ on any Lipschitz domain; this is required in the proof of Lemma 5.2 in \cite{BMW} when $p>1$ as well as for establishing \eqref{eigenvalue} in the case of $\alpha$-stable symmetric processes as per \cite{Davies}. (Note that bounded convex domains are Lipschitz; see \cite{DekelLev}.) Further, the inequality:
\begin{equation}\label{ExitTimeIneq}
    \E_x[\tau_D]\le C d(x,\partial D)^{\beta}
\end{equation}
used in Lemma 5.2 from \cite{BMW} (and hence Lemma \ref{BMWGeneral} above) holds for $\beta=\alpha/2$ and $C>0$ a constant dependent on $D$ if $\tau_D$ the first exit time of $A_t$ from a $C^{1,1}$ domain $D$ (see \cite{Kulczycki2}). Here, the bounds on the Green function of $D$ from \cite{Kulczycki2} are used to get (\ref{ExitTimeIneq}). Using this, \cite{Siudeja} obtains the same bound for $D$ a convex domain where $C$ in this case depends on the radius $r$ of the outer ball condition and the diameter of $D$. Hence if $D\in\Ocal_{\varepsilon,B}$ is convex, then since the balls for the outer ball condition can be of any radius for convex domains, let $r=\varepsilon$, and in bounding the diameter of $D$ by that of $B$, we get that $C$ is bounded above by a constant $C_{\varepsilon,B}$ dependent only on $\varepsilon$ and $B$. Although it is yet to be established if the same inequality holds for every Lipschitz domain, the above reasoning states that Lemma \ref{BMWGeneral} easily extends to an $\alpha$-stable symmetric process $A_t$ when the domains $D$ and $G$ are $C^{1,1}$ or convex; however, for what follows, we restrict our interest to only convex domains.
\end{remark}

An immediate corollary to Lemma \ref{BMWGeneral} is the following which holds for any sequence of domains $\{D_m\}\subset\Ocal_{\varepsilon,B}$ converging to some $D'$ with respect to the Hausdorff metric:
\begin{corollary}\label{LpConvergence} Given $\varepsilon>0$, if $\{D_m\}_{m=0}^\infty\subset\Ocal_{\varepsilon,B}$ converges to $D'$ with respect to the Hausdorff metric, then for any $p\in(0,1]$:
$$\sup_{x\in D'}\E_x[|\tau_{D_m}-\tau_{D'}|^p]\rightarrow0.$$
\end{corollary}
\begin{proof}
    Firstly, if $D\in\Ocal_{\varepsilon,B}$, then there exists a constant $C_{\varepsilon,B}$ dependent on $\varepsilon$ and $B$ such that the constant $C$ as in (\ref{ExitTimeIneq}) satisfies $C\le C_{\varepsilon,B}$. This inequality already holds if $\tau_D$ is the first exit time of $A_t$ from a convex domain $D$ as per Remark \ref{generalizationRemark}, so consider when $\tau_D$ is the first exit time of Brownian motion from a Lipschitz domain $D$. Then Proposition 2.3 in \cite{DeBlassie} gives that not only $\beta\le 2$, but the only dependence of $\beta$ on $D$ is on the angle $\theta$ of a uniform cone condition (as mentioned in \cite{BMW}, this proposition extends to $d\ge2$). Since $D\in\Ocal_{\varepsilon,B}$, however, let $\theta=\varepsilon$ so that $\beta$ is only dependent on $\varepsilon$. The constant $C$, on the other hand, depends only on $\theta=\varepsilon$ and the radius of the ball $B$; because of this, we have $C=C_{\varepsilon,B}$.


    Now, if $p=1$ as in Lemma \ref{BMWGeneral}, then since each of $D,G\in\Ocal_{\varepsilon,B}$, the constants $C_{1,D}$ and $C_{1,G}$ can be bounded above by a constant $C_{\varepsilon,B}$ as per the above argument. This holds when $A_t$ is either Brownian motion or an $\alpha$-stable symmetric process, $\alpha\in(0,2)$, the latter provided each of $D$, $G$ are convex. Hence for $D_m,D'\in\Ocal_{\varepsilon,B}$ (with each domain convex if $A_t$ is an $\alpha$-stable symmetric process, $\alpha\in(0,2)$), Lemma \ref{BMWGeneral} gives:
    \begin{align*}
        &\sup_{x\in D'}\E_x[|\tau_{D_m}-\tau_{D'}|]\\&\le C_{\varepsilon,B}\max\left\{\sup_{x\in D'\setminus{D_m}}(d(x,\partial D'))^\beta,\sup_{x\in D_m\setminus{D'}}(d(x,\partial D_m))^\beta\right\}.
    \end{align*}
    Since $d_\Hcal(D_m,D')\rightarrow0$, each expression inside the brackets above goes to 0, yielding the desired result.

    If $p\in(0,1)$, an application of Jensen's inequality yields:
    $$\sup_{x\in D'}\E_x[|\tau_{D_m}-\tau_{D'}|^p]\le\sup_{x\in D'}\left(\E_x[|\tau_{D_m}-\tau_{D'}|]\right)^p.$$
    Since the expression on the right hand side of the inequality converges to 0 as $m\rightarrow\infty$ as per the above argument, the result follows for $p\in(0,1)$.
\end{proof}

We now turn to the proof of Theorem \ref{MainResult2}.

\begin{proof} To see that
\begin{equation}\label{DistributionConv}
|\P_{x_m}(\tau_{D_m}>t)-\P_{x'}(\tau_{D'}>t)|\rightarrow0
\end{equation}
holds, we will break up the difference above as follows:
\begin{align*}
    &|\P_{x_m}(\tau_{D_m}>t)-\P_{x'}(\tau_{D'}>t)|\\
    &\le|\P_{x_m}(\tau_{D_m}>t)-\P_{x_m}(\tau_{D'}>t)|+|\P_{x_m}(\tau_{D'}>t)-\P_{x'}(\tau_{D'}>t)|.
\end{align*}

To show that the first expression goes to 0, note that since $x_m\rightarrow x'\in D'$, we have that $x_m\in D'$ for $m$ large enough, and so:
\begin{align*}
    |\P_{x_m}(\tau_{D_m}>t)-\P_{x_m}(\tau_{D'}>t)|\le\sup_{y\in D'}|\P_{y}(\tau_{D_m}>t)-\P_{y}(\tau_{D'}>t)|.
\end{align*}
To see that:
\begin{equation}\label{convDist1}
    \sup_{y\in D'}|\P_{y}(\tau_{D_m}>t)-\P_{y}(\tau_{D'}>t)|\rightarrow0
\end{equation}
as $m\rightarrow\infty$, we first use the joint continuity of $p_{D'}$ in $(0,\infty)\times D'\times D'$; that is, given $\varepsilon>0$, let $\delta>0$ be such that $|(t',x',y')-(t,x,y)|<\delta$ implies $|p_{D'}(t',x',y')-p_{D'}(t,x,y)|<\varepsilon/\Lcal(D')$. Then:
\begin{align}
    &\sup_{y\in D'}|\P_{y}(\tau_{D'}>t+\delta/4)-\P_{y}(\tau_{D'}>t)|\\&=\sup_{y\in D'}\left|\int_{D'}p_{D'}(t+\delta/4,z,y)dz-\int_{D'}p_{D'}(t,z,y)dz\right|\nonumber\\
    &\le\sup_{y\in D'}\int_{D'}\left|p_{D'}(t+\delta/4,z,y)-p_{D'}(t,z,y)\right|dz<\varepsilon.\label{convDist2}
\end{align}
This helps to establish that:
$$\lim_{\delta\rightarrow0^+}\sup_{y\in D'}\P_y(\tau_{D'}\le t+\delta/4)\sup_{y\in D'}\P_y(\tau_{D'}\le t)$$
and similar reasoning yields:
$$\lim_{\delta\rightarrow0^+}\sup_{y\in D'}\P_y(\tau_{D'}\le t-\delta/4)=\sup_{y\in D'}\P_y(\tau_{D'}\le t).$$
From here, we will make use of the inequalities:
\begin{equation}\label{bound1}
    \P_y(\tau_{D'}\le t-\delta/4)\le\P_y(\tau_{D_m}\le t)+\P_y(|\tau_{D_m}-\tau_{D'}|>\delta/4)
\end{equation}
and
\begin{equation}\label{bound2}
    \P_y(\tau_{D'}\le t)\le\P_y(\tau_{D_m}\le t+\delta/4)+\P_y(|\tau_{D_m}-\tau_{D'}|>\delta/4).
\end{equation}
Here, rearrange \eqref{bound1} as the following:
\begin{equation*}
    -\P_y(\tau_{D_m}\le t)\le-\P_y(\tau_{D'}\le t-\delta/4)+\P_y(|\tau_{D_m}-\tau_{D'}|>\delta/4)
\end{equation*}
then add $\P_y(\tau_{D'}\le t+\delta/4)$ to both sides so that we have:
\begin{align*}
    &\P_y(\tau_{D'}\le t+\delta/4)-\P_y(\tau_{D_m}\le t)\\
    &\le\P_y(\tau_{D'}\le t+\delta/4)-\P_y(\tau_{D'}\le t-\delta/4)+\P_y(|\tau_{D_m}-\tau_{D'}|>\delta/4)
\end{align*}
implying that:
\begin{align*}
    &\sup_{y\in D'}[\P_y(\tau_{D'}\le t+\delta/4)-\P_y(\tau_{D_m}\le t))]\\
    &\le\sup_{y\in D'}[\P_y(\tau_{D'}\le t+\delta/4)-\P_y(\tau_{D'}\le t-\delta/4)]+\sup_{y\in D'}\P_y(|\tau_{D_m}-\tau_{D'}|>\delta/4).
\end{align*}
From this, similar reasoning as in \eqref{convDist2} shows that:
$$\sup_{y\in D'}[\P_y(\tau_{D'}\le t+\delta/4)-\P_y(\tau_{D'}\le t-\delta/4)]<\varepsilon.$$
Next, since $\sup_{y\in D'}\P_y(|\tau_{D_m}-\tau_{D'}|>\delta/4)\rightarrow0$ as $m\rightarrow\infty$ by Corollary \ref{LpConvergence} and Chebyshev's Inequality, given $\varepsilon>0$, let $M\in\N$ hold such that $m\ge M$ implies that:
$$\sup_{y\in D'}\P_y(|\tau_{D_m}-\tau_{D'}|>\delta/4)<\varepsilon.$$
so that:
\begin{equation}\label{boundA}
    \sup_{y\in D'}[\P_y(\tau_{D'}\le t+\delta/4)-\P_y(\tau_{D_m}\le t))]<2\varepsilon
\end{equation}
In addition, \eqref{bound2} allows us to obtain, for $m\ge M$:
$$\sup_{y\in D'}\P_y(\tau_{D_m}\le t)-\P_y(\tau_{D'}\le t+\delta/4)<\varepsilon$$
and so this and \eqref{boundA} imply:
$$\sup_{y\in D'}|\P_y(\tau_{D_m}\le t)-\P_y(\tau_{D'}\le t+\delta/4)|<2\varepsilon.$$
Therefore we have the work below which establishes \eqref{convDist1} as $\varepsilon>0$ is arbitrary:
\begin{align*}
    &\sup_{y\in D'}|\P_{y}(\tau_{D_m}> t)-\P_{y}(\tau_{D'}> t)|\nonumber\\
    &\le\sup_{y\in D'}|\P_{y}(\tau_{D_m}> t)-\P_{y}(\tau_{D'}> t+\delta/4)|\\&\hspace{3cm}+\sup_{y\in D'}|\P_{y}(\tau_{D'}> t+\delta/4)-\P_{y}(\tau_{D'}> t)|\nonumber\\
    &=\sup_{y\in D'}|\P_{y}(\tau_{D_m}\le t)-\P_{y}(\tau_{D'}\le t+\delta/4)|+\\&\hspace{3cm}\sup_{y\in D'}|\P_{y}(\tau_{D'}>t+\delta/4)-\P_{y}(\tau_{D'}> t)|\nonumber\\
    &<3\varepsilon.\nonumber
\end{align*}

To see that $|\P_{x_m}(\tau_{D'}>t)-\P_{x'}(\tau_{D'}>t)|\rightarrow0$ as $m\rightarrow\infty$, consider the following:
\begin{align}
    |\P_{x_m}(\tau_{D'}>t)-\P_{x'}(\tau_{D'}>t)|&=\left|\int_{D'}p_{D'}(t,y,x_m)dy-\int_{D'}p_{D'}(t,y,x')dy\right|\nonumber\\
    &\le\int_{D'}\left|p_{D'}(t,y,x_m)-p_{D'}(t,y,x')\right|dy.\label{integralDifference}
\end{align}
From here, Corollary 4.8 in \cite{Bass} gives that \eqref{integralDifference} goes to 0 as $m\rightarrow\infty$, implying that:

$$\lim_{m\rightarrow\infty}|\P_{x_m}(\tau_{D'}>t)-\P_{x'}(\tau_{D'}>t)|=0$$
and hence
$$\lim_{m\rightarrow\infty}|\P_{x_m}(\tau_{D_m}>t)-\P_{x'}(\tau_{D'}>t)|=0,$$
as desired.
\end{proof}

Denoting $p_D^\alpha(t,x,y)$ as the transition density of the fractional Laplacian $-(-\frac{1}{2}\Delta)^{\alpha/2}$ in $D$, then since $p^\alpha(t,x,y)$ and $p_D^\alpha(t,x,y)$ share many properties in common with $p(t,x,y)$ and $p_D(t,x,y)$, respectively (see Theorems 2.1 and 2.4 in \cite{ChenSong}), then the same reasoning as above along with Remark \ref{generalizationRemark} shows that not only does Corollary 4.8 in \cite{Bass} hold for $p_D^\alpha$, but Theorem \ref{MainResult2} holds for $\alpha$-stable symmetric processes provided each $D_m$ is convex. This will be needed in proving Corollary \ref{applications} for $A_t$ any $\alpha$-stable symmetric process, $\alpha\in(0,2]$; for further details, see Section \ref{Appschapter}.

\begin{remark}\label{SSLipschitz}
It should be noted that Steiner symmetrization does not necessarily preserve the Lipschitz boundary. As a counterexample, consider $f(x)$ defined below:
\[   f(x)=\left\{
\begin{array}{ll}
      3-2\sqrt{1-(x-1)^2} & x\in(-2,0) \\
      3 & \text{else} \\
\end{array} 
\right. \]
and consider the open set:
$$D=\{(x,y)\in\R^2:-3<x<3,-f(-x)<y<f(x)\}.$$
Then $D$ satisfies the $\varepsilon$-cone property where $\varepsilon=1/4$; however, in performing a Steiner symmetrization with respect to the $x$-axis, then the resulting region:
$$D^\#=\{(x,y)\in\R^2:-3<x<3,-0.5[f(x)+f(-x)]<y<0.5[f(x)+f(-x)]\}$$
cannot satisfy the $\varepsilon$-cone property for any $\varepsilon>0$ due to the cusps at $(0,3)$ and $(0,-3)$. Because of this, $D^\#$ cannot be a Lipschitz domain.
\end{remark}

Now in applying Theorem \ref{MainResult2} to Corollary \ref{MainResult2-SS}, given $D_0=D$ and $x_0=x\in D$ as in Section \ref{Introduction}, let $D_m$ be the $m$-th consecutive symmetrization of $D$ with respect to the sequence of hyperplanes $\{\Lambda_m\}_{m=1}^\infty$ that converges to $D'$ with respect to the Hausdorff metric; then by Remark \ref{SSLipschitz} we must assume that each $D_m$ is in the class of domains $\Ocal_{\varepsilon,B}$ for some $\varepsilon>0$ and ball $B$. To prove Corollary \ref{MainResult2-SS}, we will show that the corresponding sequence of orthogonal projections $x_m\in D_m$ converges to some $x'\in D'$ as shown below.

\begin{lemma} Consider the sequences $\{x_m\}_{m=0}^\infty$, $\{D_m\}_{m=0}^\infty$, and $\{\Lambda_m\}_{m=1}^\infty$ described above. If $D_m$ converges to some $D'$ with respect to the Hausdorff metric, then there exists $x'$ such that $x_m\rightarrow x'$.
\end{lemma}
\begin{proof} Suppose by contradiction that $x_m$ does not converge; that is, for every $x'\in\R^d$, there exists $\varepsilon>0$ such that for every $M\ge1$, there exists $m\ge M$ such that $|x_m-x'|\ge\varepsilon$. Then for such $m\ge M$ and every $y\in B_{\varepsilon/2}(x_m)$:
$$|x'-y|=|(x'-x_m)-(y-x_m)|\ge\left||x'-x_m|-|y-x_m|\right|>\varepsilon-\varepsilon/2=\varepsilon/2$$
so that the ball $B_{\varepsilon/2}(x_m)$ cannot converge to any ball $B_{\varepsilon/2}(x')$ with respect to the Hausdorff metric. Hence if $D\subseteq B_{\varepsilon/2}(x)$, then since Steiner symmetrization reduces diameter (see Theorem 6.14, \cite{Baernstein}), we have that each $D_m\subseteq B_{\varepsilon/2}(x_m)$ for $m\ge1$; however, the choice of hyperplanes $\Lambda_m$ does not permit the sequence of balls $B_{\varepsilon/2}(x_m)$ to converge with respect to the Hausdorff metric, so the sequence of symmetrized domains $D_m$ with $x_m\in D_m$ does not converge with respect to the Hausdorff metric either. Hence we obtain an contradiction in this case. Else, if $D\not\subseteq B_{\varepsilon/2}(x)$, let $t\in(0,1)$ be such that $t(D-x)+x\subset B_{\varepsilon/2}(x)$. By the above reasoning, the corresponding sequence of domains $t(D_m-x_m)+x_m\subset B_{\varepsilon/2}(x_m)$ cannot converge with respect to the Hausdorff metric, and so neither can $D_m$ by scaling of the domains, a contradiction.
\end{proof}

An immediate corollary of this is the following establishing Corollary \ref{MainResult2-SS}.

\begin{corollary}\label{xmInD} If $d_\Hcal(D_{m},D')\rightarrow0$, then there exists $M\ge1$ such that $x_{m}\in D'$ for all $m\ge M$ and hence $x'\in D'$.
\end{corollary}

\section{Applications}\label{Appschapter}
\subsection{Triangles Converging to an Equilateral Triangle}\label{AppsTriangle}
This first part of the application section is devoted to proving Corollary \ref{applications} for the case $n=3$; more precisely, let $T\subset\R^2$ be a triangle of fixed area and $T'$ be an equilateral triangle centered at the origin $0\in\R^2$ with the same area as $T$. Then for any $x\in T$, $t>0$:
\begin{equation}\label{triangleIneq}
    \P_x(\tau^\alpha_T>t)\le\P_0(\tau^\alpha_{T'}>t),
\end{equation}
where $\tau^\alpha_T$, $\tau^\alpha_{T'}$ are the first exit times of an $\alpha$-stable symmetric process from $T$ and $T'$, respectively. Recall that we may extend to an $\alpha$-stable symmetric process due to Remark \ref{generalizationRemark} and the fact that triangles are always convex.

To prove this, we appeal to an algorithm from \cite{PolyaSzego} and \cite{Henrot} that transforms any triangle $T\subset\R^2$ of fixed area to an equilateral triangle $T'$ with the same area using a countable sequence of Steiner symmetrizations. 
More precisely, a Steiner symmetrization is performed on $T$ with respect to the mediator of one of its sides to obtain $T_1$; another symmetrization would then be done on $T_1$ with respect to the mediator of a different side yielding $T_2$; and another on $T_2$ with respect to the mediator of the remaining side to get $T_3$. This process would repeat indefinitely to get a sequence of triangles $T_m$ in $\R^2$. It was shown in \cite{Henrot} that the sine of each angle of $T_m$ converges to $\sqrt{3}/2$. This means that the sequence $T_m$ converges to an equilateral triangle $T'$ of the same area; however, we need convergence of $T_m$ with respect to the Hausdorff metric to apply Theorem \ref{MainResult2}, which \cite{Henrot} does not explicitly prove. Hence we will prove the following below.

\begin{proposition}\label{triangleProp1} Let $T_m,T'$ be as above. Then $d_\Hcal(T_m,T')\rightarrow0$.
\end{proposition}
\begin{proof} Since the sine of each angle of $T_m$ converges to $\sqrt{3}/2$, this yields that each angle of $T_m$ must converge to $\pi/3$, a characteristic unique to equilateral triangles. Further, the mediator of each side of $T'$ intersects the opposing vertex so that the Steiner symmetrization of $T'$ with respect to these mediators is itself. Because of this and the fact that each mediator of $T'$ intersects at the center of $T'$, the distance between the mediators of each side of $T_m$ and the respective opposite vertex converges to 0. 

Further, because the Steiner symmetrizations act on mediators of $T_m$, given $\eta>0$, there exists $M\ge1$ such that for ${m_1},{m_2}\ge M$:
$$\max\left\{\sup_{x\in T_{m_1}}d(x,T_{m_2}),\sup_{x\in T_{m_2}}d(T_{m_1},x)\right\}<\eta.$$
Hence, $\{T_m\}$ is a Cauchy sequence with respect to $d_\Hcal$. Further, since Steiner symmetrization decreases the diameter of a domain (see \cite{Baernstein}), they are uniformly bounded by a closed ball $B$ of radius $2\cdot\text{diam}(T_1)$. Thus with $d_\Hcal$ complete in the compact metric space $\overline{B}$ (see \cite{Henrikson}), the sequence must converge with respect to $d_\Hcal$. Since the angles of $T_m$ converge to $\pi/3$, $T_m$ must converge to an equilateral triangle $T'$ with respect to $d_\Hcal$.
\end{proof}

With this algorithm, note that the smallest angle $\alpha_m$ of $T_m$ is at most that of $T_{m+1}$; the same goes for the minimum distance $\beta_m$ between any vertex in $T_m$ and its opposite side. As a result, with $T_m$ converging to $T'$ in the Hausdorff metric, $\alpha_m$ and $\beta_m$ are increasing sequences that converge respectively to $\pi/3$ and $\sqrt[4]{3}\sqrt{A}$, where the latter is the height of $T'$ with area $A>0$. Thus, if $T_1$ satisfies the $\varepsilon$-cone property where $\varepsilon=\frac{1}{4}\min\{\alpha_1,\beta_1\}$, then so do each of $T_m$ and $T'$. In addition, each $T_m,T'$ is in the bounded ball $B$ described in the proof of Proposition \ref{triangleProp1} above, so $T_m,T'\in\Ocal_{\varepsilon,B}$.

Now, letting $x_0=x$ and $T_0=T$, Theorem \ref{MainResult1} establishes that:
$$\P_{x_0}(\tau^\alpha_{T_0}>t)\le\P_{x_1}(\tau^\alpha_{T_1}>t)\le\cdots\le\P_{x_m}(\tau^\alpha_{T_m}>t),$$
where $x_m\in T_m$ in this case denotes the $m$-th consecutive orthogonal projection of $x_0=x$. Further, we have that:
$$\P_x(\tau^\alpha_T>t)\le\lim_{m\rightarrow\infty}\P_{x_m}(\tau^\alpha_{T_m}>t)$$
so since $\P_{x_m}(\tau^\alpha_{T_m}>t)$ is a nondecreasing bounded sequence, it must converge to a finite probability. From here, to prove (\ref{triangleIneq}), we need to establish the following using Corollary \ref{MainResult2-SS}:
\begin{equation*}
    \lim_{m\rightarrow\infty}\P_{x_m}(\tau^\alpha_{T_m}>t)=\P_{0}(\tau^\alpha_{T'}>t)
\end{equation*}
so to do this, we will prove that $x_m\rightarrow0$ below.

\begin{lemma} Let $x_m\in T_m$ be the orthogonal projections described above. Then $x_m\rightarrow0$.    
\end{lemma}

\begin{proof} As in the algorithm in \cite{Henrot}, we will first perform a countable number of consecutive symmetrizations on $T'$ such that each line of symmetrization connects between 0 and a vertex of $T'$. More precisely, let $l_i$ be the line connecting 0 and the vertex $v_i$, $i=1,2,3$; then the first three symmetrizations will be performed with respect to $l_1$, $l_2$, and $l_3$, respectively, and from there, the process repeats indefinitely. We still have $T'$ after each symmetrization, but denoting the orthogonal projection at the $m$-th step as $x_m'$, we first  claim that $x_m'\rightarrow0$.

To see this, first note that the intersection of all three lines $l_1$, $l_2$, $l_3$ is precisely at 0. Hence, if $x'\in T'$ is such that the first symmetrization with respect to $l_1$ causes the orthogonal projection $x'_{1}$ to be the center 0, then we are done since $0^\#=0$ for symmetrization performed with respect to any $l_i$. Else, if $x_1'\neq0$, then $x'_{1}$ is not orthogonal to the slice of $T'$ that is perpendicular to $l_2$ and has 0, so $x'_{2}\neq0$ either, and by similar iterative reasoning, none of the other $x'_{m}$ points equal 0 either. On the other hand, though, each $x'_{m}$ is also distinct; this can be seen by drawing a right triangle with one vertex at $x'_{m}$, the other at its orthogonal projection $x'_{m+1}$ with respect to the mediator on which the Steiner symmetrization is performed, and the other at 0. From this right triangle, we can see that $|x'_{m+1}|=|x'_{m}|\sin(\pi/6)$ for every $m\ge1$ and hence $x_m'\rightarrow0$.


To see that $x_m\rightarrow0$ from this, since $d_\Hcal(T_m,T')\rightarrow0$, we have that each $x_m$ can be approximated by $x_m'$ above; that is, given $\eta>0$, there exists $M_1\ge1$ such that $x_m\in T_m$ and the respective $x_m'\in T_m'$ satisfy $|x_m-x_m'|<\eta$ for $m\ge M_1$. Also, since $x_m'\rightarrow0$, let $M_2\ge1$ satisfy $|x'_m|<\eta$ for $m\ge M_2$. Hence for $m\ge\max\{M_1,M_2\}$:
\begin{align*}
    \left|x_{m}\right|&\le\left|x_{m}-x_{m}'\right|+\left|x_{m}'\right|<2\eta
\end{align*}
and so $|x_m|<2\eta$. Hence with $\eta>0$ arbitrary, the sequence $x_{m}$ converges to $0\in T'$.
\end{proof}

\subsection{Quadrilaterals Converging to a Square}
We now consider proving Corollary \ref{applications} in the case $n=4$; more precisely, let $Q\subset\R^2$ be a quadrilateral of fixed area and $Q'$ be a square centered at $0\in\R^2$ with the same area as $Q$. Then for $x\in Q, t>0$:
\begin{equation}\label{quadIneq}
\P_x(\tau^\alpha_Q>t)\le\P_0(\tau^\alpha_{Q'}>t),
\end{equation}
where $\tau^\alpha_Q,\tau^\alpha_{Q'}$ are first exit times of an $\alpha$-stable symmetric process from $Q$ and $Q'$, respectively.

For this case, first note that \cite{Henrot} pictorially gives an algorithm used to transform any quadrilateral $Q$ into a rectangle $R$ with the same area via three Steiner symmetrizations; however, to get to the square, we require an algorithm from \cite{PolyaSzego} which transforms a rectangle into a square using a countable number of Steiner symmetrizations. For the convenience of the reader, we will describe the algorithm transforming $Q$ into $R$ from \cite{Henrot} and then prove how the algorithm in \cite{PolyaSzego} creates a sequence of convex quadrilaterals $Q_m$ with $Q_0=R$ that converge to the square $Q'$ with respect to the Hausdorff metric provided each line of symmetrization has the origin $0\in\R^2$.

\begin{proposition}\label{finiteQuad} Given any quadrilateral $Q\subset\R^2$, it takes at most 3 Steiner symmetrizations to transform $Q$ into a rectangle $R$ with the same area.
\end{proposition}
\begin{proof} Let us first consider the case when $Q$ is a parallelogram. Without loss of generality, let $Q$ here be such that a pair of parallel lines are also parallel to the $x$-axis. Here, the line at which the Steiner symmetrization occurs must be perpendicular to the parallel lines chosen; then a symmetrization with respect to the line $l=\{(x,y)\in\R^2:x=0\}$ will transform $Q$ into a rectangle centered at the origin in $\R^2$, as the midpoints of the parallel lines will be on $l$ and the other two lines will also become parallel to each other with the same sides; thus we get a rectangle.

Next, consider the case when $Q$ is a kite (with 2 pairs of adjacent congruent sides). Consider the line segment $l$ whose endpoints are the vertices of the kite where the congruent sides meet; then the line of symmetrization is perpendicular to $l$ and intersects $l$ at its midpoint. The quadrilateral then becomes a parallelogram, at which point we refer to the process above to get that at most 2 Steiner symmetrizations are required if $Q$ is a kite.

Finally, in the most general case, consider the two line segments formed by the opposite vertices of $Q$. For the line segment of longer length $l$, the line at which the Steiner symmetrization occurs must be perpendicular to $l$ and intersect this line at its midpoint. The quadrilateral $Q$ then becomes a kite, so referring to the second paragraph above, we get that at most 3 Steiner symmetrizations are needed to transform $Q$ into a rectangle of the same area.
\end{proof}

Because of the above proposition, we only need to apply Theorem \ref{MainResult1} at most three times to get that, for any $w\in Q$, $t>0$:
$$\P_w(\tau^\alpha_Q>t)\le\sup_{z\in R}\P_z(\tau^\alpha_R>t),$$
where $\tau^\alpha_R$ is the first exit time from $R$ of an $\alpha$-stable symmetric process. Further, let $R$ be centered at the origin in $\R^2$ with two sides parallel to the $x$-axis and the other two parallel to the $y$-axis. Then two Steiner symmetrizations on the $x$- and $y$-axes respectively yield that:
\begin{equation*}
\P_{(x,y)}(\tau^\alpha_R>t)\le\P_{(0,y)}(\tau^\alpha_R>t)\le\P_{(0,0)}(\tau^\alpha_R>t)
\end{equation*}
and so for any $w\in Q$, $t>0$:
\begin{equation*}
    \P_w(\tau^\alpha_Q>t)\le\P_0(\tau^\alpha_R>t).
\end{equation*}

To prove (\ref{quadIneq}) from here, it suffices to show:
\begin{equation}\label{quadIneq2}
    \P_0(\tau^\alpha_R>t)\le\P_0(\tau^\alpha_{Q'}>t).
\end{equation}
To establish \eqref{quadIneq2}, recall from \cite{PolyaSzego} the algorithm that transforms $R$ into $Q'$; first symmetrize $R$ with respect to a line perpendicular to one of its diagonals to obtain a rhombus $Q_1$, then symmetrize $Q_1$ with respect to a line perpendicular to one of its sides to get a rectangle $Q_2$. Repeat these two steps indefinitely to get a sequence of quadrilaterals $Q_m$. Thus to establish \eqref{quadIneq2}, we will apply Corollary \ref{MainResult2-SS} with $x_m,x'=0$ after showing that $d_\Hcal(Q_m,Q')\rightarrow0$, which we prove below.

\begin{proposition}\label{quadAlgorithm} Let $Q_m$ and $Q'$ be as above. If the rectangle $R$ above is centered at the origin $0\in\R^2$, then $d_\Hcal(Q_m,Q')\rightarrow0$.
\end{proposition}
\begin{proof}
    Without loss of generality, let $R$ have vertices:
    $$P_1=(a,b),\text{ }P_2=(-a,b),\text{ }P_3=(-a,-b),\text{ }P_4=(a,-b),$$
    and consider the line $l_1=\{(x,y)\in\R^2:y=-\frac{a}{b}x\}$ perpendicular to the diagonal through $P_1$ and $P_3$. Performing a symmetrization on $R$ with respect to $l_1$ yields a rhombus $Q_1$ with vertices at:
    \begin{equation}\label{rhombusPts}
    P'_1=P_1,\text{ }P'_2=\left(\frac{-2ab^2}{a^2+b^2},\frac{2a^2b}{a^2+b^2}\right),\text{ }P'_3=P_3,\text{ }P'_4=\left(\frac{2ab^2}{a^2+b^2},\frac{-2a^2b}{a^2+b^2}\right),
    \end{equation}
and each side of $Q_1$ has length:
$$\frac{\sqrt{a^6+7a^4b^2+7a^2b^4+b^6}}{a^2+b^2}.$$
Note that $Q_1$ is still centered at the origin since the diagonals of $Q_1$ intersect there. Also, since $0\in l_1$, we have that $0^\#=0$.

To get to the rectangle $Q_2$ from $Q_1$, the symmetrization will be performed about the line $l_2$ below:
$$l_2=\left\{(x,y)\in\R^2:y=-\frac{a^3+3ab^2}{b^3-a^2b}x\right\}$$
which is perpendicular to the line segments $\overline{Q_1Q_2}$ and $\overline{Q_3Q_4}$. This forms a rectangle $Q_2$ centered at the origin with side lengths:
\begin{equation}\label{newSides}
b'=4ab\sqrt{\frac{a^2+b^2}{a^4+6a^2b^2+b^4}}\hspace{.5cm}\text{and}\hspace{.5cm}a'=\sqrt{\frac{a^4+6a^2b^2+b^4}{a^2+b^2}}.
\end{equation}
To see that the difference between the side lengths has decreased, consider the recursive definitions below based off of (\ref{newSides}):
\begin{equation*}
b_{m+1}=4a_mb_m\sqrt{\frac{a_m^2+b_m^2}{a_m^4+6a_m^2b_m^2+b_m^4}}\hspace{.5cm}\text{and}\hspace{.5cm}a_{m+1}=\sqrt{\frac{a_m^4+6a_m^2b_m^2+b_m^4}{a_m^2+b_m^2}}.
\end{equation*}
We will next look at the quotient $b_m/a_m$ in the following way:
$$\frac{b_{m+1}}{a_{m+1}}=4a_mb_m\left(\frac{a_m^2+b_m^2}{a_m^4+6a_m^2b_m^2+b_m^4}\right)=\frac{b_m}{a_m}\left(\frac{4a_m^2(a_m^2+b_m^2)}{a_m^4+6a_m^2b_m^2+b_m^4}\right)$$
so that if we let $c_m=b_m/a_m$, then we can rewrite the above recursive relation as:
$$c_{m+1}=\frac{4c_m+4c_m^3}{1+6c_m^2+c_m^4}.$$
From this, consider the function:
$$f(c)=\frac{4c^3+4c}{c^4+6c^2+1}.$$
Using methods from Calculus, one can see that $f(c)$ increases on $(0,1)$ and decreases on $(1,\infty)$. Hence with $f(1)=1$, we have that the sequence $c_m$ satisfies $c_0\in(0,\infty)$ and $c_m<1$ for all $m\ge1$. This yields that $c_m$ is a bounded increasing sequence for $m\ge 1$, so it will converge to a fixed point of $f(c)$; however, for $c>0$, the only positive fixed point of $f(c)$ is 1, so it must hold that $b_m/a_m=c_m\rightarrow1$.

Now, with the sides of the rectangles in $Q_m$ converging to the same length, the distance between the line of symmetrization $y=-\frac{a_m}{b_m}x$ and the endpoints of the diagonal formed by $(a_m,-b_m)$ and $(-a_m,b_m)$ converge to 0 as $m\rightarrow\infty$. Further, by considering the points of a rhombus $Q_m$ with vertices of the form described in (\ref{rhombusPts}), the fact that $b_m/a_m\rightarrow1$ yields that the angle at any vertex on the rhombus converges to $\pi/2$; thus the rhombi in $Q_m$ also converge to a square of the same area. With these in mind, in a similar fashion as in the triangle case in Section \ref{AppsTriangle}, for given $\eta>0$, there exists $M\ge1$ such that the quantities $\sup_{x\in Q_{m_1}}d(x,Q_{m_2})$ and $\sup_{x\in Q_{m_2}}d(Q_{m_1},x)$ are both bounded by $\eta$ for ${m_1},{m_2}\ge M$. Hence, $\{Q_m\}$ is a Cauchy sequence with respect to $d_\Hcal$ in a compact metric space, and so the sequence must converge with respect to $d_\Hcal$. Since the ratio of the sides of $Q_m$ converge to 1, $Q_m$ must converge to a square $Q'$ with respect to $d_\Hcal$.
\end{proof}
The above proof helps to establish that:
$$\P_0(\tau^\alpha_R>t)\le\P_0(\tau^\alpha_{Q_1}>t)\le\P_0(\tau^\alpha_{Q_2}>t)\le\cdots$$
since the orthogonal projection of the origin in each symmetrization is itself (with the origin being on each line of symmetrization). Hence the sequence $x_m=0$ is constant and so converges to $0\in\R^2$. To see that each $Q_m$ satisfies the $\varepsilon$-cone property, first note that the minimum side length $l_m$ of each rectangle in $Q_m$ increases to the square root of the area of $Q_m$. For the rhombi in $Q_m$, the smaller angle $\alpha_m$ in $Q_m$ is at most that of $Q_{m+2}$, the next rhombus in the sequence; the same goes for the distance $\beta_m$ between the larger angle of $Q_m$ and the line segment connecting the vertices with the smaller angle. This yields that $\alpha_m$ and $\beta_m$ increase respectively to $\pi/2$ and $\sqrt{A/2}$ where $A$ is the area of $Q'$. Hence, if $Q_1$ and $Q_2$ satisfy the $\varepsilon$-cone property where $\varepsilon=\frac{1}{4}\min\{l_1,\alpha_2,\beta_2\}$, then so does each $Q_m$. Thus, using Corollary \ref{MainResult2-SS}, we have (\ref{quadIneq2}), and hence (\ref{quadIneq}), as desired.

\section*{Acknowledgements}
The author would like to thank Professor Rodrigo Ba\~nuelos for suggesting the problem as well as his valuable insights and time while preparing this paper. Author was supported in part by NSF Grant \#DMS-1854709 under PI Rodrigo Ba\~nuelos and is part of the author's PhD Thesis.

\bibliography{bibtex}{}
\bibliographystyle{plain}

\end{document}